\documentclass[reqno]{amsart}

\usepackage{amssymb}
\usepackage[all]{xy}

\newcommand\frg{\ensuremath{\mathfrak{g}}} \newcommand{\g}{\frg}
\newcommand\frh{\ensuremath{\mathfrak{h}}}

\newcommand\lsl{\ensuremath{\mathfrak{sl}}}
\newcommand\NN{\mathbb{N}}
\newcommand\QQ{\mathbb{Q}}

\newcommand\ZZ{\mathbb{Z}}
\newcommand\al{\alpha}

\newcommand\la{\lambda}

\newcommand\om{\omega} 
\newcommand\lan{\langle} \newcommand\ran{\rangle}

\newcommand\ch{\sp{\scriptscriptstyle\vee}}

\newcommand{\CH}{\operatorname{CH}}

\DeclareMathOperator{\Span}{Span} 
\DeclareMathOperator{\tr}{tr} 

\DeclareMathOperator{\rk}{rank}
\swapnumbers \theoremstyle{plain}
\newtheorem{thm}{Theorem}[section]
\newtheorem{lem}[thm]{Lemma}
\newtheorem{cor}[thm]{Corollary}
\newtheorem{prop}[thm]{Proposition}

\theoremstyle{definition}

\newtheorem{ex}[thm]{Example}
\newtheorem{rem}[thm]{Remark}
\newtheorem{dfn}[thm]{Definition}

\newcommand{\zz}{\mathbb{Z}}   % integers

\newcommand{\cc}{\mathfrak{c}} % characteristic map
\newcommand{\Lb}{\mathcal{L}}   % line bundle

\title[Polynomial invariants and fundamental representations]{Basic polynomial invariants, fundamental representations and the Chern class map}

\author{S.~Baek}
\author{E.~Neher}
\author{K.~Zainoulline}

\address{Sanghoon Baek, Department of Mathematics and Statistics,
University of Ottawa}
\email{sbaek@uottawa.ca}

\address{Erhard Neher, Department of Mathematics and Statistics,
University of Ottawa}
\email{erhard.neher@uottawa.ca}

\address{Kirill Zainoulline, Department of Mathematics and Statistics,
University of Ottawa,
         585 King Edward, Ottawa ON K1N 6N5, Canada}
\email{kirill@uottawa.ca}


\begin{document}

\maketitle

%%% Document STARTS here

%%%%%%%%%%%%%%%%%%%%%%%%%%%%%%%%%
%%%%%%%%%%%%%%%%%%%%%%%%%%%%%%%%%
%%%%%%%%%%%%%%%%%%%%%%%%%%%%%%%%%

\section*{Introduction}

Consider a crystallographic root system $\Phi$ together with its Weyl group $W$ acting on the weight lattice $\Lambda$ of $\Phi$. Let $\ZZ[\Lambda]^W$ and $S^*(\Lambda)^W$ be the $W$-invariant subrings of the integral group ring
$\ZZ[\Lambda]$ and the symmetric algebra $S^*(\Lambda)$.
A celebrated theorem of Chevalley says that $\ZZ[\Lambda]^W$ is a polynomial
ring over $\ZZ$ in classes of fundamental representations $\rho_1,\ldots,\rho_n$
and $S^*(\Lambda)^{W}\otimes\mathbb{Q}$ is a polynomial ring over $\mathbb{Q}$ in basic polynomial invariants
$q_1,\ldots,q_n$, where $n=\rk(\Phi)$.

\smallbreak

In the present paper we establish and investigate the relationship between $\rho_i$'s
and $q_i$'s. To do this we introduce an equivariant analogue of the Chern class map $\phi_i$ that
provides an isomorphism between the truncated rings $\ZZ[\Lambda]/I_m^j$ and $S^*(\Lambda)/I_a^j$ modulo
powers of the respective augmentation ideals. This allows us
to express basic polynomial invariants in terms of fundamental representations and vice versa, hence relating
the geometry of the variety of Borel subgroups $X$ with representation theory of the respective Lie algebra $\g$.

\smallbreak

A multiple of $\phi_i$ restricted to the respective cohomology ($K_0$ and $CH^*$) of $X$
gives the classical Chern class map $c_i\colon K_0(X)\to CH^i(X)$.
This geomeric interpretation provides a powerful tool to compute the annihilators of the torsion of the Grothendieck
$\gamma$-filtration on $K_0$ of twisted forms of $X$ as well as a tool to estimate the torsion part of its Chow groups in small codimensions.

\smallbreak

The paper is organized as follows. In the first section we introduce the $I$-adic filtrations
on $\ZZ[\Lambda]$ and $S^*(\Lambda)$ together with an isomorphism $\phi_i$ on their truncations.
Then we study the subrings of invariants and introduce the key notion of an exponent $\tau_i$
of a $W$-action on a free abelian group $\Lambda$. Roughly speaking, the integers $\tau_i$ measure how far is the ring $S^*(\Lambda)^W$ (with integer coefficients)
from being a polynomial ring in $q_i$'s. In section 5 we compute all the exponents up to degree 4 and
show that they all coincide with the Dynkin index of the Lie algebra $\g$.
In section 6 we apply the obtained results to estimate the torsion in Grothendieck $\gamma$-filtration
of some twisted flag varieties. In section 7 we compute the second exponent $\tau_2$ for
a non-crystallographic group $H_2$. 
\smallbreak

\paragraph{\bf Acknowledgments.} The first author has been partially supported from the NSERC grants of the other two authors and from the Fields Institute. The second author gratefully acknowledges support through NSERC Discovery grant 8836-20121. The last author has been supported by the NSERC Discovery grant 385795-2010, Accelerator Supplement 396100-2010 and an Early Researcher Award (Ontario).

%%%%%%%%%%%%%%%%%%%%%%%%%%%%%%%%%
%%%%%%%%%%%%%%%%%%%%%%%%%%%%%%%%%
%%%%%%%%%%%%%%%%%%%%%%%%%%%%%%%%%

\section{Two filtrations}

Consider the two covariant functors $S^*(-)$ and $\ZZ[-]$
from the category of abelian groups to the category of commutative rings
$$
S^*(-)\colon \Lambda\mapsto S^*(\Lambda)
\text{ and }
 \ZZ[-]\colon \Lambda\mapsto  \ZZ[\Lambda]
$$
given by taking the symmetric algebra
of an abelian group $\Lambda$ and
the integral group ring of $\Lambda$ respectively.
The $i$th graded component $S^i(\Lambda)$
is additively generated by monomials
$\lambda_1\lambda_2\ldots \lambda_i$
with $\lambda_j \in \Lambda$ and
the ring $ \ZZ[\Lambda]$ is additively generated
by exponents $e^\lambda$, $\lambda\in \Lambda$.

\smallbreak

The trivial group homomorphism induces the ring homomorphisms
$$
\epsilon_a\colon S^*(\Lambda)\to  \ZZ
\text{ and }
\epsilon_m\colon \ZZ[\Lambda]\to  \ZZ
$$
called the augmentation maps.
By definition $\epsilon_a$ sends every element of positive degree to $0$ and
$\epsilon_m$ sends every $e^\lambda$ to $1$.
Let $I_a$ and $I_m$ denote the kernels
of $\epsilon_a$ and $\epsilon_m$ respectively.
Observe that $I_a=S^{>0}(\Lambda)$ consists of elements of positive degree
and $I_m$ is generated by differences $(1-e^{-\lambda})$, $\lambda\in \Lambda$.
Consider the respective $I$-adic filtrations:
$$
S^*(\Lambda)=I_a^0\supseteq I_a \supseteq I_a^2\supseteq\ldots
\text{ and }
\ZZ[\Lambda]=I_m^0\supseteq I_m \supseteq I_m^2\supseteq\ldots
$$
and let
$$
gr_a^*(\Lambda)=\bigoplus_{i\ge 0} I_a^{i}/I_a^{i+1}
\text{ and }
gr_m^*(\Lambda)=\bigoplus_{i\ge 0}I_m^i/I_m^{i+1}
$$
denote the associated graded rings.
Observe that  $gr_a^*(\Lambda)= S^*(\Lambda)$.

\begin{ex}\label{onevar}
If $\Lambda\simeq  \ZZ$,
then the ring $S^*(\Lambda)$
can be identified with the polynomial ring
in one variable $ \ZZ[\omega]$,
where $\omega$ is a generator of $\Lambda$
and the ring $ \ZZ[\Lambda]$
can be identified with the Laurent polynomial ring
$ \ZZ[x,x^{-1}]$ where $x=e^\omega$.
The augmentations $\epsilon_a$ and $\epsilon_m$ are given by
$$
\epsilon_a\colon \omega\mapsto 0
\text{ and }
\epsilon_m\colon x\mapsto 1.
$$
We have
$I_a=(\omega)$ and
$I_m$ is additively generated by differences $(1-x^n)$, $n\in  \ZZ$.

\smallbreak

Note that the rings $ \ZZ[\omega]$ and $ \ZZ[x,x^{-1}]$ are not isomorphic,
however they become isomorphic after the truncation.
Namely for every $i\ge 0$ there is ring isomorphism
$$
\phi_i\colon   \ZZ[x,x^{-1}]/I_m^{i+1}
\stackrel{\simeq}\to  \ZZ[\omega]/I_a^{i+1}
$$
defined by
$\phi_i\colon x\mapsto (1-\omega)^{-1}=1+\omega+\ldots +\omega^i$
with the inverse defined by $\phi_i^{-1}\colon \omega \mapsto 1-x^{-1}$.
It is useful to keep the following picture in mind
$$
\xymatrix{
\ZZ[x,x^{-1}] \ar@{>>}[d] \ar[rd] & \ZZ[\omega] \ar@{>>}[d] \ar[l] \\
\ZZ[x,x^{-1}]/I_m^{i+1} \ar[r]^{\phi_i}_{\simeq} & \ZZ[\om]/I_a^{i+1}
}
$$
observing that the inverse $\phi_i^{-1}$
can be lifted to the map $\ZZ[\omega] \to \ZZ[x,x^{-1}]$
but $\phi_i$ can't.
\end{ex}

The example can be generalized as follows:

\begin{lem}\cite[2.1]{GaZa}
Assume that $\Lambda$ is a free abelian group of finite rank $n$. The rings $ \ZZ[\Lambda]$ and $S^*(\Lambda)$
become isomorphic after truncation.
Namely, if $\{\omega_1,\ldots,\omega_n\}$
is a $ \ZZ$-basis of $\Lambda$,
then for every $i\ge 0$ there is a ring isomorphism
$$
\phi_i\colon \ZZ[\Lambda]/I_m^{i+1}
\stackrel{\simeq}\to S^*(\Lambda)/I_a^{i+1}
$$
defined by $\phi_i(1)=1$ and
$$
\phi_i(e^{\sum_{j=1}^n
a_j\omega_j})= \prod_{j=1}^n(1-\omega_j)^{-a_j}
$$
with the inverse defined by $\phi_i^{-1}(\omega_j)=1-e^{-\omega_j}$.
\end{lem}

Note that the map $\phi_i$ preserves the $I$-adic filtrations.
Indeed, by definition $\phi_i(I_m^j)\subseteq I_a^j$ for every $0\le j\le i$.
Moreover, we have the following

\begin{lem}\label{canmap}
The isomorphism $\phi_i$ restricted to the subsequent quotients $I_m^i/I_m^{i+1}$
doesn't depend on the choice of a basis of $\Lambda$.
Hence, there is an induced canonical isomorphism of graded rings
$$
\phi_*=\oplus_{i\ge 0}\phi_i \;\colon gr_m^*(\Lambda)
\stackrel{\simeq}\longrightarrow gr_a^*(\Lambda)=S^*(\Lambda).
$$
\end{lem}

\begin{proof}
Indeed, in this case we can define the inverse
$\phi_i^{-1}\colon I_a^i/I_a^{i+1} \to I_m^i/I_m^{i+1}$
by
$$
\phi_i^{-1}(\lambda_1\lambda_2\ldots\lambda_i)=
(1-e^{-\lambda_1})(1-e^{-\lambda_2})\ldots (1-e^{-\lambda_i}).
$$
It is well-defined since
$(1-e^{-\lambda-\lambda'})=(1-e^{-\lambda})+ (1-e^{-\lambda'})$
modulo $I_m^2$.
\end{proof}

Consider the composite of the map $\phi_i$ with the projections
$$
\phi^{(i)}\colon\zz[\Lambda] \to \ZZ[\Lambda]/I_m^{i+1}
\stackrel{\phi_i}\longrightarrow
S^*(\Lambda)/I_a^{i+1} \to S^i(\Lambda).
$$
The map $\phi^{(i)}$, and therefore $\phi_i$,
can be computed on generators
$e^\lambda$, $\lambda\in \Lambda$ as follows:

\smallbreak

Let $f(z)=\prod_j(1-\omega_j z)^{-a_j}$,
where $\lambda=\sum_j a_j\omega_j$.
Then
$$
\phi^{(i)}(e^{\sum_j a_j\omega_j})=
\frac{1}{i!}\frac{d^i f(z)}{d z^i}\Big|_{z=0}
$$
To compute the derivatives of $f(z)$ we observe that $f'(z)=f(z)g(z)$,
where $g(z)=\sum_j a_j\omega_j (1-\omega_j z)^{-1}$ and
$
\frac{d^i\, g(z)}{d\, z^i}=
\sum_{j}\frac{i!\,a_{j}\omega_{j}^{i+1}}{(1-\omega_{j}z)^{i+1}}.
$
Hence, starting with $g_0=1$ we obtain the following recursive formulas
$$
\frac{d^i\, f(z)}{d\, z^i}=f(z)\cdot g_{i}(z),
\text{ where }
g_i(z)=g(z)g_{i-1}(z)+g_{i-1}'(z).
$$

\begin{ex}\label{formulas}
For small values of $i$ we obtain

\smallbreak

\begin{tabular}{r|l}
$i$ & $i!\cdot \phi^{(i)}(e^\lambda)=$ \\ \hline
1 & $\lambda$ \\
2 & $\lambda^2 + \lambda(2) $ \\
3 & $\lambda^3 + 3\lambda(2)\lambda+2\lambda(3)$ \\
4 & $\lambda^4 + 6\lambda(4) + 6\lambda(2)\lambda^2+8\lambda(3)\lambda+3\lambda(2)^2$
\end{tabular}

\smallbreak

\noindent
where given a presentation
$\lambda=\sum_{j=1}^n a_{j,\lambda}\omega_j$, $a_{j,\lambda}\in \zz$ in terms
of the basis $\{\om_1,\om_2,\ldots.\om_n\}$, the character $\lambda(m)$, $m\ge 1$ is defined by
$$
\lambda(m) = \sum_{j=1}^n a_{j,\lambda}\omega_j^m.
$$
\end{ex}

%%%%%%%%%%%%%%%%%%%%%%%%%%%%%%%%%%%%%%%%
%%%%%%%%%%%%%%%%%%%%%%%%%%%%%%%%%%%%%%%%
%%%%%%%%%%%%%%%%%%%%%%%%%%%%%%%%%%%%%%%%

\section{Invariants and exponents}

Let $W$ be a finite group which acts on a free abelian group $\Lambda$ of finite rank
by $\zz$-linear automorphisms.
Consider the induced action of $W$ on $\zz[\Lambda]$ and $S^*(\Lambda)$.
Observe that it is compatible with the $I$-adic filtrations,
i.e. $W(I_m^i)\subseteq I_m^i$ and
$W(I_a^i)\subseteq I_a^i$ for every $i\ge 0$.

\smallbreak

Note that the isomorphisms $\phi_i$ and $\phi_i^{-1}$
are not necessarily $W$-equivariant.
However, by Lemma~\ref{canmap}
their restrictions to the subsequent quotients $I_m^i/I_m^{i+1}$
and $I_a^i/I_a^{i+1}=S^i(\Lambda)$ are $W$-equivariant
and we have
$$
(I_m^i/I_m^{i+1})^W\simeq (I_a^i/I_a^{i+1})^W.
$$

\smallbreak

Let $I_m^W$ denote the ideal of $\zz[\Lambda]$
generated by $W$-invariant elements
from the augmentation ideal $I_m$,
i.e., by elements from $\zz[\Lambda]^W\cap I_m$.
Similarly, let $I_a^W$ denote the ideal of $S^*(\Lambda)$
generated by $W$-invariant elements from $I_a$,
i.e., by elements from $S^*(\Lambda)^W\cap I_a$.
%Hence, for the quotient rings we have
%$$
%\zz[\Lambda]/I_m^W \simeq \zz\otimes_{\zz[\Lambda]^W}\zz[\Lambda]
%\text{ and }
%S^*(\Lambda)/I_a^W \simeq \zz\otimes_{S^*(\Lambda)^W} S^*(\Lambda).
%$$

\smallbreak

For each $\chi\in \Lambda$
let $\rho(\chi)=\sum_{\lambda\in W(\chi)} e^\lambda$
denote the sum over all elements of the $W$-orbit of $\chi$.
Every element in $I_m^W$ can be written as
a finite linear combination with integer coefficients
of the elements $\hat\rho(\chi)=\rho(\chi)-\epsilon_m(\rho(\chi))$, $\chi\in \Lambda$. Therefore, the ideal $I_m^W$ is generated by the  elements $\hat \rho(\chi)$, i.e.,
$$
I_m^W=\langle \hat\rho(\chi)\mid \chi\in \Lambda\rangle.
$$

\smallbreak

The image of $I_m^W$ by means of the composite
$$
\zz[\Lambda]\to \zz[\Lambda]/I_m^{i+1}
\stackrel{\phi_i}\longrightarrow
S^*(\Lambda)/I_a^{i+1}\, .
$$
is an ideal in $S^*(\Lambda)/I_a^{i+1}$
generated by the elements $\phi_i(\hat \rho(\chi))$, $\chi\in \Lambda$. Therefore, the image of $I_m^W$ in $S^i(\Lambda)$
is the $i$th homogeneous component of the ideal
generated by $\phi^{(j)}(\hat \rho(\chi))$,
where $1\le j\le i$, $\chi \in \Lambda$, i.e.
$$
\phi^{(i)}(I_m^W)=
\langle f\cdot \phi^{(j)}(\hat \rho(\chi))
\mid 1\le j \le i,\; f\in S^{i-j}(\Lambda),\; \chi\in \Lambda\rangle_\zz\, .
$$

We are ready now to introduce the central notion of the present paper:

\begin{dfn}\label{conjecture}
We say that an action of $W$ on $\Lambda$
\emph{has finite exponent in degree $i$} if
there exists a non-zero integer $N_i$ such that
$$
N_i\cdot (I_a^W)^{(i)} \subseteq \phi^{(i)}(I_m^W),
$$
where $(I_a^W)^{(i)}=I_a^W\cap S^i(\Lambda)$.
In this case the g.c.d. of all such $N_i$s
will be called the \emph{$i$-th exponent} of the $W$-action
and will be denoted by $\tau_i$.
\end{dfn}

Observe that if $\phi^{(i)}(I_m^W)$ is a subgroup
of finite index in $(I_a^W)^{(i)}$,
then $\tau_i$ is simply the exponent of
$\phi^{(i)}(I_m^W)$ in $(I_a^W)^{(i)}$.
Note also that by the very definition
$\tau_0=1$ and $\tau_i\mid \tau_{i+1}$ for every $i\ge 0$.

%%%%%%%%%%%%%%%%%%%%%%%%%%%%%%%%%%
%%%%%%%%%%%%%%%%%%%%%%%%%%%%%%%%%%
%%%%%%%%%%%%%%%%%%%%%%%%%%%%%%%%%%

\section{Essential actions}

In the present section we study $W$-actions
that have no $W$-invariant linear forms,
i.e. we assume that $\Lambda^W=0$.
In the theory of reflection groups such actions
are called {\em essential} (see \cite[V, \S3.7]{Bo} or \cite{Hum}).
Note that this immediately implies that $\tau_1=1$.

\begin{lem}\label{twlambda}
For every $\chi\in \Lambda$ and $m\in \NN_+$ we have
$\sum_{ \la \in W(\chi)} \la(m) = 0$.
\end{lem}

\begin{proof}
Let $\om_1,\om_2,\ldots.\om_n$ be a $\zz$-basis of $\Lambda$. For $m\in \NN_+$ we have
$$
\sum_{\lambda \in W(\chi)} \lambda(m)
 =\sum_{\lambda \in W(\chi)}
       \big( \sum_{j=1}^n a_{j,\lambda}\omega_j^m \big)
= \sum_{j=1}^n \big( \sum_{\lambda \in W(\chi)}
         a_{j,\lambda} \big) \omega_j^m .
$$
In particular, for $m=1$ we obtain
$$
\sum_{\lambda \in W(\chi)} \lambda = \sum_{j=1}^n \big(
\sum_{\lambda \in W(\chi)}
         a_{j,\lambda} \big) \omega_i .
$$
Since $\Lambda^W=0$,
we have $\sum_{\lambda \in W(\chi)} \lambda=0$.
Since $\om_j$, $1\le j \le n$ are $\ZZ$-free,
we have $\sum_{\lambda \in W(\chi)}a_{j,\lambda} = 0$
for all $1\le j \le n$.
\end{proof}

\begin{cor}\label{phi2}
For every $\chi\in \Lambda$ we have
$$
\phi^{(2)}(\rho(\chi))=
\tfrac{1}{2}{\sum_{\lambda \in W(\chi)} \lambda^2}.
$$
In particular, the quadratic form $\phi^{(2)}(\rho(\chi))$ is $W$-invariant, i.e.
$$
\phi^{(2)}(\rho(\chi))\in S^2(\Lambda)^W.
$$
\end{cor}

\begin{proof}
By the formula for $\phi^{(2)}$ in Example \ref{formulas} and by Lemma~\ref{twlambda}
we obtain that
$$
\phi^{(2)} \big( \sum_{\lambda\in W(\chi)}e^\lambda \big)=
\tfrac{1}{2} \sum_{\lambda\in W(\chi)} (\lambda^2+ \lambda(2))=
\tfrac{1}{2}{\sum_{\lambda \in W(\chi)} \lambda^2}.
\qedhere
$$
\end{proof}

\begin{cor}\label{phi2cor}
If $S^2(\Lambda)^W=\langle q\rangle$ for some $q$,
then $\phi^{(2)}(I_m^W)$ is a subgroup of finite index in $(I_a^W)^{(2)}$.
\end{cor}

\begin{proof}
The image of the ideal $I_m^W$ is generated by
$\phi^{(1)}(\rho(\chi))$ and $\phi^{(2)}(\rho(\chi))$.
Since $\Lambda^W=0$,
$\phi^{(1)}(\rho(\chi))=\sum_{\lambda\in W(\chi)} \lambda=0$
and by Corollary~\ref{phi2}, $\phi^{(2)}(I_m^W)$ is generated only by
the $W$-invariant quadratic forms $\phi^{(2)}(\rho(\chi))$.
For every $\chi\in \Lambda$ let
\begin{equation} \label{def:N}
\phi^{(2)}(\rho(\chi))=N_\chi \cdot q,\; N_\chi\in \NN.
\end{equation}
Then the subgroup $\phi^{(2)}(I_m^W)$
is a subgroup of $(I_a^W)^{(2)}$ of exponent
$$
\tau_2=\gcd_{\chi\in \Lambda} N_\chi.
\qedhere
$$
\end{proof}

We now investigate the invariants of degree $3$ and $4$.

\begin{lem}\label{phi3}
For every $\chi\in \Lambda$ we have
$$
\phi^{(3)}(\rho(\chi))=
\tfrac{1}{6} \sum_{\lambda \in W(\chi)} (\lambda^3 + 3\lambda(2)\lambda).
$$
\end{lem}

\begin{proof}
By the formula for $\phi^{(3)}$ in Example \ref{formulas} and by Lemma~\ref{twlambda}
we obtain that
$$
\phi^{(3)}(\rho(\chi))=
\tfrac{1}{6} \sum_{\lambda \in W(\chi)}
(\lambda^3 + 3\lambda(2)\lambda + 2\lambda(3))=
\tfrac{1}{6} \sum_{\lambda \in W(\chi)}
(\lambda^3 + 3\lambda(2)\lambda).
\qedhere
$$
\end{proof}

\begin{lem}\label{phi4}
For every $\chi\in \Lambda$ we have
$$
\phi^{(4)}(\rho(\chi))=
\tfrac{1}{24} \sum_{\lambda \in W(\chi)}
[\lambda^4 +6\lambda(2)\lambda^{2}+8\lambda(3)\lambda+3\lambda(2)^{2}].
$$
\end{lem}

\begin{proof}
It follows from Example~\ref{formulas} and Lemma~\ref{twlambda}.
\end{proof}

%%%%%%%%%%%%%%%%%%%%%%%%%%%%%%%%%%%%
%%%%%%%%%%%%%%%%%%%%%%%%%%%%%%%%%%%%
%%%%%%%%%%%%%%%%%%%%%%%%%%%%%%%%%%%%

\section{The Dynkin index}

In the present section we show that
the action of the Weyl group $W$ of a crystallographic root system $\Phi$ on the weight lattice $\Lambda$ has finite exponent in degree 2 
which coincides with the Dynkin index of the respective Lie algebra.

\smallbreak

Let $W$ be the Weyl group of a crystallographic root system $\Phi$ and let
$\Lambda$ be its weight lattice as defined in \cite[\S2.9]{Hum}.
Let $\{\omega_1,\ldots,\omega_n\}$ be a basis of $\Lambda$
consisting of fundamental weights (here $n$ is the rank of $\Phi$).

\smallbreak

The Weyl group $W$ acts on $\lambda\in \Lambda$
by means of simple reflections
$$
s_j(\lambda)=
\lambda - \langle \alpha_j^\vee,\lambda \rangle \cdot \alpha_j,\quad j=1\ldots n
$$
where $\alpha_j^\vee$ is the $j$-th simple coroot
and $\langle -,- \rangle$ is the usual pairing.
Note that $\langle \alpha_j^\vee,\omega_i\rangle = \delta_{ij}$,
where $\delta_{ij}$ is the Kronecker symbol.

\smallbreak

The subring of invariants $\ZZ[\Lambda]^{W}$ is the representation ring of
the respective Lie algebra $\g$.
By a theorem of Chevalley it is the polynomial ring in fundamental representations
$\rho(\omega_j)\in \zz[\Lambda]^W$, i.e.
$$
\zz[\Lambda]^W\simeq \zz[\rho(\omega_1),\ldots,\rho(\omega_n)].
$$
Observe that the dimension of the fundamental representation
$\rho(\omega_j)$ equals to the number of elements in the orbit
that is $\epsilon_m(\rho(\omega_j))$.

\smallbreak

Therefore, the ideal $I_m^W$ is generated by the elements
$\hat\rho(\omega_j)$, $j=1\ldots n$
and its image $\phi^{(i)}(I_m^W)$
is the $i$-th homogeneous component of the ideal
generated by $\phi^{(j)}(\rho(\omega_l))$, $1\le j\le i$, $l=1\ldots n$.

\begin{lem}\label{deg1}
We have $\Lambda^W=0$ and hence also
$$
\phi^{(1)}(\zz[\Lambda]^W)=\phi^{(1)}(I_m^W)=0.
$$
\end{lem}

\begin{proof}
Let $\eta \in \Lambda^W$.
Since $\eta = s_{\al_j}(\eta) = \eta -\lan \eta, \al_j\ch\ran \al_j$
we have
$\lan \eta ,\al_j\ch\ran =\frac{2(\alpha_j,\eta)}{(\alpha_j,\alpha_j)} = 0$
for all simple roots $\al_j$ which implies that $\eta=0$.
\end{proof}

\begin{lem}
We have $S^2(\Lambda)^{W}=\langle q\rangle$.
\end{lem}

\begin{proof}
By \cite[Prop.~4]{gross-nebe} there exists
an integer valued $W$-invariant quadratic form on $\Lambda$
which has value $1$ on short coroots.
As the group $S^2(\Lambda)^{W}$ is identical to the group
of all integral $W$-invariant quadratic forms on $T_{*}\otimes \mathbb{R}$,
the result follows.
\end{proof}

\begin{cor}\label{conj2}
The image $\phi^{(2)}(I_m^W)$
is a subgroup of $(I_a^W)^{(2)}$ of finite index.
\end{cor}

\begin{proof}
This follows from Corollary~\ref{phi2cor} and Lemma \ref{deg1}.
\end{proof}

We recall briefly the notion of indices of representations
introduced by Dynkin \cite[\S2]{Dynkin} (See also \cite{Braden}).

\smallbreak

Let
$f:\g\to \g'$ be a morphism between Lie
algebras. Then there exists a unique number $j_{f}\in \mathbb{C}$,
called the \emph{Dynkin index of $f$}, satisfying
\[ (f(x),f(y))=j_{f}(x,y),\]
for all $x,y\in \g$, where (--,--) is the Killing form on
$\g$ and $\g'$ normalized such that $(\alpha,\alpha)=2$ for any long root
$\alpha$.
In particular, if $f:\g\to\lsl(V)$ is a linear representation, $j_f$ is a positive
integer, called the \emph{Dynkin index of the linear representation
$f$}, defined by
\[\tr(f(x),f(y))=j_{f}(x,y).\]

\smallbreak

The \emph{Dynkin index of $\g$} is defined to be the greatest common
divisor of all the Dynkin indices of all linear representations of $\g$. By \cite[(2.24) and
(2.25)]{Dynkin}, the Dynkin index of $\g$ is the greatest common divisor
of the Dynkin index of its fundamental representations. Moreover,
all the Dynkin indices of the fundamental representations were
calculated in \cite[Table 5]{Dynkin}.

\smallbreak

Using the $\lsl_{2}$-representation theory, the Dynkin index of a linear
representation $f:\g\to \lsl(V)$ can be described as follows. Let $\alpha$ be a long root.
For the formal character ${\rm ch}(V) =
\sum_{\la} n_{\la} e^\la$, one has (see \cite[Lemma 2.4]{LaSo} or \cite[5.1 and
Lemma~5.2]{kumar-nara-rama})
\begin{equation}\label{Dynkinsl}
j_{f}=\frac{1}{2} \sum _{\la} \lan \la, \al\ch\ran^2.
\end{equation}

\begin{thm}\label{lem:chev2} The integers
$N(\om_j)$ for the $j$-th fundamental weight as defined in
\eqref{def:N} coincide with the Dynkin index of the fundamental
representation with highest weight $\om_j$. In particular, the
second exponent $\tau_2$ coincides with the Dynkin index of $\g$.
\end{thm}

\begin{proof}
To find the precise value of $\tau_2$ we use the explicit formula
for $\phi^{(2)}$, that is
$$
\phi^{(2)}(\rho(\chi))=\tfrac{1}{2}\sum_{\lambda\in W(\chi)}
\lambda^2.
$$
We know  that $\tau_2$ is the greatest common
divisor of the integers $N_j =N_{\om_j}$ using the notation of the
proof of Corollary~\ref{phi2cor}, where  $\om_j$ is the $j$-th
fundamental weight of $\g$.
%Since by \cite[Prop.~2.3]{BoKu} also
As the Dynkin index is the greatest common divisor of the Dynkin indices of
the fundamental representations $\om_j$, it suffices to show that
$N_j$ coincides with the Dynkin index of the representation $V_j$
corresponding to $\om_j$. We can view $\phi^{(2)}(\rho(\chi))$ for
$\chi=\om_j$ as a function on the lattice $\frh_\ZZ =
\Span_\ZZ\{\al\ch \mid \al \in \Phi \mbox{ long}\}$. Since $V_j$ has
character ${\rm ch}(V_j) = \sum_{\la \in W(\om_j)} e^\la$, by (\ref{Dynkinsl}) the Dynkin index of the
representation $V_j$ is $\tfrac{1}{2} \sum _{\la \in W(\om_j)} \lan
\la, \al\ch\ran^2$, where $\al$ is any long root in $\Phi$. Thus,
$\phi^{(2)}(\rho(\om_j))$ is the constant function with value $N_j$.
\end{proof}

We note that a different proof of Theorem~\ref{lem:chev2} was given in
\cite[\S2]{GaZa}.

%%%%%%%%%%%%%%%%%%%%%%%%%%%%%%%%%%
%%%%%%%%%%%%%%%%%%%%%%%%%%%%%%%%%%%%
%%%%%%%%%%%%%%%%%%%%%%%%%%%%%%%%%%%%

\section{Exponents of degrees 3 and 4}

In the present section we show that $\tau_2=\tau_3=\tau_4$ 
for all crystallographic root systems

\smallbreak

Let $S=\{\la_1,\ldots,\la_r\}$ be a finite set of weights. 
We denote by $-S$ the set of opposite
weights $\{-\la_1,\ldots,-\la_r\}$, by $S_+$ the set of sums $\{\lambda_i+\lambda_j\}_{i<j}$,
by $S_-$ the set of differences $\{\lambda_i-\lambda_j\}_{i<j}$ and by $S_\pm$ the disjoint
union $S_+\amalg S_-$. By definition we have $|S_+|=|S_-|=\binom{r}{2}$.

Using the fact that $(\lambda+\lambda')(m)=\lambda(m)+\lambda'(m)$ for every $\lambda,\lambda'\in \Lambda$ and $m\ge 0$ we obtain the following lemma which will be extensively used in the computations

\begin{lem}\label{rest4}(i) For every integer $m_1, m_2, x, y\ge 0$ and a finite subset $S\subset\Lambda$ we have
$$
\sum_{\lambda\in S\amalg -S} \lambda(m_1)^x\lambda(m_2)^y = 
(1+(-1)^{x+y})\sum_{\lambda\in S}\lambda(m_1)^x\lambda(m_2)^y.
$$
In particular, $\sum_{\lambda \in S\amalg -S} \lambda(2)\lambda^2=0$.

(ii) For every subset $S\subset \Lambda$ with $|S|=r$ and for every $m_1,m_2\ge 0$ we have
$$
\sum_{\lambda\in S_+}\lambda(m_1)\lambda(m_2)=(r-1)\sum_{\lambda\in S} \lambda(m_1)\lambda(m_2) + \sum_{i\neq j} \lambda_i(m_1)\lambda_j(m_2)\text{ and }
$$
$$
\sum_{\lambda\in S_-}\lambda(m_1)\lambda(m_2)=(r-1)\sum_{\lambda\in S} \lambda(m_1)\lambda(m_2) - \sum_{i\neq j} \lambda_i(m_1)\lambda_j(m_2).
$$
In particular, this implies that $\sum_{\lambda\in S_\pm}\lambda(m_1)\lambda(m_2)=2(r-1)\sum_{\lambda\in S}\lambda(m_1)\lambda(m_2)$.
\end{lem}

\paragraph{\bf $A_n$-case}
Let $\Phi$ be of type $A_n$ for $n\geq 3$. We denote the canonical
basis of $\mathbb{R}^{n+1}$ by $e_{i}$ with $1\leq i \leq n+1$. 
According to  \cite[\S3.5 and \S3.12]{Hum} the basic polynomial invariants of 
the $W$-action on $\Lambda$
(algebraically independent homogeneous generators of $S^*(\Lambda)^{W}$ as a $\QQ$-algebra)
are given by the symmetric power sums
$$
q_{i}:=e_{1}^{i}+\cdots
+e_{n+1}^{i},\quad 2\leq i \leq n+1.
$$

\smallbreak

Let $s_i$ denote the $i$th elementary symmetric function in $e_1,\ldots,e_{n+1}$. Using
the classical identities
$$
q_1=s_1,\quad q_i=s_1q_{i-1}-s_2q_{i-2}+\ldots +(-1)^i s_{i-1}q_1+(-1)^{i+1}i\cdot s_i,\quad 1< i<n+1
$$
and the fact that $s_1=0$, we obtain that
$$
q_2/2=-s_2,\; q_3/3=s_3,\text{ and } q_4/2=s_2^2-2s_4.
$$
generate (with integral coefficients)
the ideal $I_a^W$ up to degree $4$.

\smallbreak

The fundamental weights of $\Phi$ can be expressed as follows 
$$
\omega_1=e_1,\;\omega_2=e_1+e_2,\;\;\ldots\;,\omega_{n-1}=e_1+\ldots +e_{n-1},\;\omega_n=-e_{n+1},
$$
where $e_1+e_2+\ldots +e_{n+1}=0$.
The orbits of $\omega_1$, $\omega_1+\omega_n$, $\omega_n$ and $\omega_2$, $\omega_{n-1}$
under the action of the Weyl group $W=S_{n+1}$ 
are given by 
$$
W(\omega_1)=\{e_1,\ldots,e_{n+1}\}=-W(\omega_n),\; 
W(\omega_1+\omega_n)=\{e_i-e_j\}_{i\neq j}\text{ and } 
$$
$$
W(\omega_2)=\{e_i+e_j\}_{i<j}=-W(\omega_{n-1}).$$
Therefore, $W(\omega_1+\omega_n)=S_-\amalg -S_-$ and $W(\omega_2)=S_+$, where $S=W(\omega_1)$.

\smallbreak

Applying Lemma~\ref{phi4} and Lemma~\ref{rest4} we obtain that
$$
\phi^{(4)}(\rho(\omega_1)+\rho(\omega_n))=
\tfrac{1}{12}\sum_{\lambda\in S}(\lambda^4+8\lambda(3)\lambda + 3\lambda(2)^2)\text{ and }
$$
$$
\phi^{(4)}(\rho(\omega_1+\omega_n)+\rho(\omega_2)+\rho(\omega_{n-1}))=
\tfrac{1}{24}\sum_{\lambda\in S_\pm \amalg -S_\pm}(\lambda^4+8\lambda(3)\lambda + 3\lambda(2)^2)=
$$
$$
=\tfrac{1}{24}\sum_{\lambda\in S_\pm \amalg -S_\pm}\lambda^4 + 
\tfrac{n}{6}\sum_{\lambda\in S}(8\lambda(3)\lambda + 3\lambda(2)^2).
$$
Then the difference
$$
\phi^{(4)}(\rho(\omega_1+\omega_n)+\rho(\omega_2)+\rho(\omega_{n-1}))-2n\cdot \phi^{(4)}(\rho(\omega_1)+\rho(\omega_n))=
$$
\begin{equation}\label{aneq}
=\tfrac{1}{24}\sum_{\lambda\in S_\pm \amalg -S_\pm}\lambda^4-\tfrac{n}{6}\sum_{\lambda\in S}\lambda^4=
\end{equation}
is a symmetric function in $e_1,\ldots, e_{n+1}$ and, therefore,
it can be always written as a polynomial in $q_i$s.
Indeed, since
$$
\sum_{\lambda\in S_\pm\amalg -S_\pm} \lambda^4=2\sum_{i<j}((e_i+e_j)^4+(e_i-e_j)^4)=4n\sum_{\lambda\in S} \lambda^4 + 24\sum_{i<j}e_i^2e_j^2,
$$
the difference \eqref{aneq} equals
$$=\sum_{i<j} e_i^2e_j^2=(q_2^2-q_4)/2.$$

\begin{lem}\label{Antype} For a root system of type $A_n$, $n\ge 2$, we have $\tau_2=\tau_3=\tau_4=1$.
\end{lem}

\begin{proof}
It is enough to show that the generators 
$q_2/2$, $q_3/3$ and $q_4/2$ are in the
ideal generated by the image of $\phi^{(i)}$, $i\le 4$.

\smallbreak

By Corollary~\ref{phi2} we have $\phi^{(2)}(\rho(\omega_1))=\tfrac{1}{2}\sum_{\lambda\in S} \lambda^2=q_2/2$.  By Lemma~\ref{phi3} we have $q_3/3=\phi^{(3)}(\rho(\omega_1))-\phi^{(3)}(\rho(\omega_n))$ (see also \cite[\S1C]{GaZa}).
If $\Phi$ is of type $A_2$, then $s_4=0$ and, hence, $q_4=q_2^2/2$. If $\Phi$ is of type $A_n$, $n\ge 3$, then by
\eqref{aneq} the generator
$q_4/2$ belongs to the ideal generated by the images of $\phi^{(2)}$ and $\phi^{(4)}$. 
\end{proof}

\begin{lem}\label{lem:chev3}
For any crystallographic root system $\Phi$ 
the third exponent $\tau_3$ of the $W$-action coincides with $\tau_2$ (the Dynkin index).
\end{lem}

\begin{proof}
If $\Phi$ is of type $A_n$, this follows from Lemma~\ref{Antype}; for the other types
there are no basic polynomial invariants of degree
$3$ \cite[\S3.7 Table 1]{Hum}. Therefore, $\tau_{3}=\tau_{2}$.
\end{proof}

\paragraph{\bf $B_n$, $C_n$ and $D_n$ cases}

Let $\Phi$ be of type $B_n$ or $C_n$ for $n\geq 2$ or of type $D_n$ for $n\ge 4$. 
We denote the canonical
basis of $\mathbb{R}^{n}$ by $e_{i}$ with $1\leq i \leq n$. 
By  \cite[\S3.5 and \S3.12]{Hum} the basic polynomial invariants of 
the $W$-action on $\Lambda$
are given by even power sums
$$
q_{2i}:=e_{1}^{2i}+\cdots
+e_{n}^{2i},\quad 1\leq i \leq n.
$$

\smallbreak

The
first two fundamental weights of $\Phi$ are given by
$\omega_1=e_1$, $\omega_2=e_1+e_2$ and their $W$-orbits are
$$
W(\omega_1)=\{\pm e_1,\ldots,\pm e_n\}\text{ and }W(\omega_2)=\{\pm e_i\pm e_j\}_{i<j}.
$$
Hence $W(\omega_1)=S\amalg -S$ and $W(\omega_2)=S_\pm\amalg -S_\pm$, where $S=\{e_1,\ldots,e_n\}$.

\smallbreak

Applying Lemma~\ref{phi4} and Lemma~\ref{rest4} we obtain that
$$
\phi^{(4)}(\rho(\omega_1))=\tfrac{1}{12}\sum_{\lambda\in S}\lambda^4 + 
\tfrac{1}{12}\sum_{\lambda\in S}(8\lambda(3)\lambda + 3\lambda(2)^2)\text { and }
$$
$$
\phi^{(4)}(\rho(\omega_2))=\tfrac{1}{24}\sum_{\lambda\in S_\pm\amalg -S_\pm}\lambda^4 + 
\tfrac{n-1}{6}\sum_{\lambda\in S} (8\lambda(3)\lambda + 3\lambda(2)^2).
$$
Then similar to the $A_n$-case we obtain
\begin{equation}\label{ortheq}
\phi^{(4)}(\rho(\omega_2))-2(n-1)\phi^{(4)}(\rho(\omega_1))=(q_2^2-q_4)/2,
\end{equation}
where $q_i=e_1^i+\ldots + e_n^i$.

\begin{lem}\label{Bntype} For a root system of type $B_n$ or $C_n$, $n\ge 2$ or $D_n$, $n\ge 4$ we have $\tau_4=\tau_2$.
\end{lem}

\begin{proof}
It is enough to show that $q_4/2$ is in the ideal generated by the image of $\phi^{(2)}$ and $\phi^{(4)}$. 

By Corollary~\ref{phi2} we have $\phi^{(2)}(\rho(\omega_1))=\sum_{\lambda\in S} \lambda^2=q_2$. Therefore, by \eqref{ortheq} 
$$q_4/2=(q_2/2)\cdot \phi^{(2)}(\rho(\omega_1))-\phi^{(4)}(\rho(\omega_2))+2(n-1)\phi^{(4)}(\rho(\omega_1)) $$
and the proof is finished.
\end{proof}

\begin{thm} For any crystallographic root system $\Phi$ we have $\tau_2=\tau_3=\tau_4$.
\end{thm}

\begin{proof}
The equality $\tau_2=\tau_3$ is proven in Lemma~\ref{lem:chev3}.
If $\Phi$ is of type $A_n$, $\tau_4=1$ follows from Lemma~\ref{Antype}.
If $\Phi$ is of type $B_n$, $C_n$ or $D_n$, $\tau_4=\tau_2$ follows from Lemma~\ref{Bntype}.
For all other types $\tau_{4}=\tau_{2}$ since there are no basic polynomial invariants of degree
$3$ and $4$ (see \cite[\S3.7 Table 1]{Hum}).
\end{proof}

%%%%%%%%%%%%%%%%%%%%%%%%%%%%
%%%%%%%%%%%%%%%%%%%%%%%%%%%%
%%%%%%%%%%%%%%%%%%%%%%%%%%%%

\section{Torsion in the Grothendieck $\gamma$-filtration}

The goal of the present section is to provide geometric interpretation (see \eqref{cherndiag})
of the map $\phi_i$
and the exponents $\tau_i$.

\smallbreak

Let $G$ be a simple simply-connected Chevalley group over a field $k$.
We fix a maximal split torus $T$ of $G$ and a Borel subgroup $B\supset T$.
Let $\Lambda$ be the group of characters of $T$.
Since $G$ is simply-connected,
$\Lambda$ coincides with the weight lattice of $G$.

\smallbreak

Let $X$ denote the variety of Borel subgroups of $G$ (conjugate to $B$).
Consider the Chow ring $\CH^*(X)$
of algebraic cycles modulo rational equivalence
and the Grothendieck ring $K_0(X)$.
Following \cite[\S1]{De74} to every character $\lambda\in \Lambda$
we may associate the line bundle $\Lb(\lambda)$ over $X$.
It induces the ring homomorphisms (called the characteristic maps)
$$
\cc_a \colon S^*(\Lambda) \to \CH^*(X)
\text{ and }
\cc_m \colon  \ZZ[\Lambda]\twoheadrightarrow K_0(X)
$$
by sending $\lambda\mapsto c_1(\Lb(\lambda))$
and $e^\lambda\mapsto [\Lb(\lambda)]$ respectively.
Note that the map $\cc_a$ is an isomorphism in codimension one, hence, giving
$$
\cc_a\colon S^1(\Lambda)=\Lambda\stackrel{\simeq}\to Pic(X)=CH^1(X)
$$
and the map $\cc_m$ is surjective. Let $W$ be the Weyl group and let $I_a^W$
and $I_m^W$ denote the respective $W$-invariant ideals.
Then according to \cite[\S4~Cor.2,\S9]{De73} and \cite[\S6]{CaPeZa}
\begin{equation}\label{invar}
\ker\cc_m=I_m^W
\end{equation}
and $\ker\cc_a$ is generated by elements of $S^*(\Lambda)$ such that their multiples are in $I^W_a$.

\smallbreak

Consider the Grothendieck $\gamma$-filtration on $K_0(X)$ (see \cite[\S1]{GaZa}).
Its $i$th term is an ideal generated by products
$$
\gamma^i(X):=
\langle (1-[\Lb_1^\vee])(1-[\Lb_2^\vee])\cdot \ldots \cdot (1-[\Lb_i^\vee])\rangle,
$$
where $\Lb_1,\Lb_2,\ldots,\Lb_i$ are line bundles over $X$.
Consider the $i$th subsequent quotient $\gamma^i(X)/\gamma^{i+1}(X)$.
The usual Chern class $c_i$
induces a group homomorphism
$c_i\colon \gamma^i(X)/\gamma^{i+1}(X) \to CH^i(X)$.

\begin{prop}
For every $i\ge 0$ there is a commutative diagram of group homomorphisms
\begin{equation}\label{cherndiag}
\xymatrix{
I_m^i/I_m^{i+1} \ar[rr]^{(-1)^{i-1}(i-1)!\cdot \phi_i}\ar@{>>}[d]^{\cc_m} & & S^i(\Lambda)\ar[d]^{\cc_a} \\
\gamma^i (X)/\gamma^{i+1}(X) \ar[rr]^{c_i} & & \CH^i(X)
}
\end{equation}
\end{prop}

\begin{proof}
Indeed, the $\gamma$-filtration on $K_0(X)$
is the image of the $I_m$-adic filtration on $\ZZ[\Lambda]$, i.e.
$\gamma^i(X)=\cc_m(I_m^i)$ for every $i\ge 0$. The Proposition then follows from the identity
$$
c_i\Big((1-[\Lb_1^\vee])(1-[\Lb_2^\vee]) \ldots (1-[\Lb_i^\vee])\Big)
=(-1)^{i-1}(i-1)!\cdot c_1(\Lb_1)c_1(\Lb_2) \ldots c_1(\Lb_i),
$$
where $\Lb_1,\Lb_2,\ldots,\Lb_i$ are line bundles over $X$ and $\Lb_i^\vee$ denotes the dual of $\Lb_i$.
\end{proof}

\begin{rem}
Note that $\ZZ[\Lambda]$ can be identitfied
with the $T$-equivariant $K_0$ of a point $pt=Spec\; k$ and $S^*(\Lambda)$
with the $T$-equivariant $CH$ of a point (see \cite{GiZa}).
The maps $\cc_a$ and $\cc_m$ then can be identified
with the pull-backs $K_T(pt)\to K_T(G)$ and $CH_T(pt)\to CH_T(G)$
induced by the structure map $G\to pt$.

\smallbreak

In view of these identifications the map $\phi_i$
can be viewed as an equivariant analogue of the Chern class map $c_i$.
\end{rem}

Consider the diagram ~\eqref{cherndiag} with $\mathbb{Q}$-coefficients.
In this case the Chern class map $c_i$ will become an isomorphism (by the Riemann-Roch theorem), the characteristic map $\cc_a$ will turn into a surjection and the map $(-1)^{i-1}(i-1)!\cdot \phi_i$ will
be an isomorphism as well.
In view of \eqref{invar} we obtain an isomorphism 
$$\phi^{(i)}\otimes\mathbb{Q}\colon I_m^W\cap I_m^i/I_m^W\cap I_m^{i+1}\otimes\mathbb{Q} \longrightarrow (I_a^W)^{(i)}\otimes\mathbb{Q}$$ 
on the kernels of $\cc_m$ and $\cc_a$.
By the very definition of the exponents $\tau_i$ this implies that

\begin{cor} The action of the Weyl group of a crystallograhic root system has finite exponent $\tau_i$ for every $i$.
\end{cor}

We are now ready to prove the main result of this section  

\begin{thm}\label{thm:torsion}
The integer $\tau_i \cdot (i-1)!$ annihilates the torsion of the $i$th subsequent
quotient
$\gamma^{i}(X)/\gamma^{i+1}(X)$ of the $\gamma$-filtration on $K_0(X)$ for $i=3,4$.
\end{thm}

\begin{rem} Note that by \cite[Expos\'e XIV, 4.5]{SGA6} for groups of types $A_n$ and $C_n$ the quotients $\gamma^{i}(X)/\gamma^{i+1}(X)$ have no torsion.
\end{rem}

\begin{proof}
Assume that $\alpha$ is a torsion element in
$\gamma^{i}(X)/\gamma^{i+1}(X)$. Then $c_i(\alpha)=0$ since $CH^i(G/B)$ has
no torsion. Let $\tilde\alpha$ be a preimage of $\alpha$ via $\cc_m$  in
$I_m^i/I_m^{i+1}\subseteq \ZZ[\Lambda]/I_m^{i+1}$. By the same analysis as in \cite[\S1B, \S1C]{GaZa} one can show that
$\ker(\cc_a)^{(i)}=(I_a^W)^{(i)}$ for $i\le 4$. By \eqref{cherndiag} we obtain that
$$
(i-1)!\,\phi_i(\tilde \alpha) \in (I_a^W)^{(i)}
$$
By definition of the index $\tau_i$ we
have
$$
\tau_i \cdot (i-1)!\, \phi_{i}(\tilde\alpha)=\phi_i(\beta), \text{
where } \beta\in I_m^W/I_m^{i+1}\cap I_m^W.
$$
Applying $\phi_i^{-1}$ to the both sides we obtain
$$
\tau_i \cdot (i-1)! \cdot \tilde\alpha = \beta \in I_m^W/I_m^{i+1}\cap I_m^W
$$
Applying $\cc_m$ to the both sides
and observing that $I_m^W=\ker \cc_m$ we obtain that
$\tau_i\cdot (i-1)!\cdot \alpha =0$.
\end{proof}

Let ${}_\xi X$ be a twisted form of the variety $X$ by means of a cocycle $\xi\in Z^1(k,G)$.
By \cite[Thm. 2.2.(2)]{Panin} the restriction map $K_0({}_\xi X) \to K_0(X)$
(here we identify $K_0(X)$ with the $K_0(X\times_k \bar k)$ over the algebraic closure $\bar k$)
is an isomorphism. Since the characteristic classes commute with restrictions, this induces
an isomorphism between the $\gamma$-filtrations, i.e. $\gamma^i({}_\xi X) \simeq \gamma^i(X)$ for every $i\ge 0$,
and between the respective quotients
$$
\gamma^i({}_\xi X)/\gamma^{i+1}({}_\xi X)\simeq \gamma^i(X)/\gamma^{i+1}(X)\quad\text{ for every }i\ge 0.
$$

In view of this fact Theorem~\ref{thm:torsion}  imply that

\begin{cor}\label{gtor} Let $G$ be a split simple simply connected group of type $B_n$ ($n\ge 3$) or $D_n$ ($n\ge 4$). Then
for every $\xi \in Z^1(k,G)$ the torsion in
$\gamma^4({}_\xi X)/\gamma^5({}_\xi X)$ is annihilated by $12$.
\end{cor}

Consider the topological filtration on $K_0(Y)$ given by the ideals
$$
\tau^i(Y):=\langle [\mathcal{O}_V]\mid V\hookrightarrow Y,\, codim_V Y\ge i \rangle.
$$
It is known (see \cite[\S2]{GaZa}) that $\gamma^i(Y)\subseteq \tau^i(Y)$ for every $i\ge 0$.

\begin{cor}\label{cor:Bn4Dn4CH4}
In the notation of Corollary~\ref{gtor} assume in addition that the induced map
$$\gamma^4({}_\xi X)/\gamma^5({}_\xi X)\to \tau^4({}_\xi X)/\tau^5({}_\xi X)$$
is surjective. Then the $2$-torsion of $\CH^{4}({}_\xi X)$ is annihilated
by $8$.
\end{cor}

\begin{proof}
 By the Riemann-Roch theorem \cite[Ex.15.3.6]{Fulton}, the composition
\[\CH^{4}({}_\xi X)\twoheadrightarrow  \tau^4({}_\xi X)/\tau^5({}_\xi X)
   \stackrel{c_{4}}\to \CH^{4}({}_\xi X)\]
is the multiplication by $(-1)^{4-1}(4-1)!=-6$, where the first map
is surjective. Hence, the torsion subgroup of
$\CH^{4}({}_\xi X)$ is annihilated by $72$ and so the result
follows.
\end{proof}

\section{`The Dynkin index' in the $H_2$ case}

 Note that the notion of an exponent $\tau_i$ can be defined over a unique factorisation domain in the same way. As an example we compute the second exponent $\tau_2$ for the action of the Weyl group of of the non-crystallographic root system $H_2$ over the base ring $\ZZ[\tfrac{1+\sqrt{5}}{2}]$, hence, giving rise to an interesting question about its geometric/Lie algebra interpretation.

\begin{thm}\label{lem:H2}
For the non-crystallographic root system $H_{2}:=I_{2}(5)$, the
second exponent $\tau_2$ is $\sqrt{5}$.
\end{thm}

\begin{proof}
We follow the notations in \cite{CMP}. In the root system $H_{2}$,
the Weyl group $W$ is the dihedral group of order $10$ and $M$ is
the $ \ZZ[\tau]$-lattice generated by two simple roots $\alpha_1$ and
$\alpha_{2}$, where $\tau=(1+\sqrt{5})/2$.
Observe that $\ZZ[\tau]$ is an Euclidean domain.

\smallbreak

The dual basis $\{\omega_1, \omega_2\}$ is defined by
\begin{equation*}
\left\{
\begin{array}{l}
\omega_{1}=\frac{1}{3-\tau}(2\alpha_1+\tau\alpha_2)\\
\omega_{2}=\frac{1}{3-\tau}(\tau\alpha_1+2\alpha_2)\\
\end{array} \right.
\text{ or } \left\{
\begin{array}{l}
\alpha_{1}=2\omega_{1}-\tau\omega_2\\
\alpha_{2}=-\tau\omega_{1}+2\omega_{2}\\
\end{array} \right.
\end{equation*}
One computes the orbits of $\omega_1$ and $\omega_2$ as follows:
\begin{align*}
W(\omega_1)&=\{\omega_1, -\omega_{2}, -\omega_{1}+\tau\omega_{2}, -\tau\omega_{1}+\omega_{2}, \tau\omega_{1}-\tau\omega_{2}\},\\
W(\omega_2)&=-W(\omega_1).
\end{align*}
As the action of $W$ on $M$ is essential, by Corollary \ref{phi2},
we have
\begin{align}\label{tau2:H21}
\phi^{(2)}(\rho(\omega_{2}))=\phi^{(2)}(\rho(\omega_{1}))&=\frac{1}{2}(\omega_1^{2}+\omega_2^{2}+(\omega_1-\tau\omega_2)^{2}+(\tau\omega_1-\omega_2)^{2}+(\tau\omega_1-\tau\omega_2)^{2})\nonumber\\
&=(1+\tau^2)\omega_{1}^{2}+(1+\tau^2)\omega_{2}^{2}-(2\tau+\tau^2)\omega_{1}\omega_{2}.
\end{align}
Since $\phi^{(2)}(\rho(\omega_2))$ is $W$-invariant by
Corollary~\ref{phi2}, we have
$$
\tau_2=\gcd(1+\tau^2,2\tau+\tau^2)=\gcd(2+\tau,2\tau-1).
$$
But $2\tau-1=\sqrt{5}$ is a prime in $ \ZZ[\tau]$, and we have
$2+\tau=(2\tau-1)\tau$ proving that
$\tau_2=\sqrt{5}$.
\end{proof}

\bibliographystyle{plain}

\end{document}